\documentclass[11pt]{amsart}

\usepackage{amssymb,amsmath}\usepackage{stmaryrd}
\usepackage{amsthm}

\usepackage[latin1]{inputenc}
\usepackage{subfigure}
\usepackage{bbm}
\usepackage{graphicx}

\newtheorem{theorem}{Theorem}
\newtheorem{lemma}[theorem]{Lemma}

\newtheorem{cor}[theorem]{Corollary}

\newcommand{\beq}{\begin{equation}}
\newcommand{\eeq}{\end{equation}}
\newcommand{\barray}{$$\begin{array}{r@{\,}c@{\,}l}}
\newcommand{\earray}{\end{array}$$}
\newcommand{\beqarray}[1]{\begin{equation}\label{#1}\begin{array}{rcl}}
\newcommand{\eeqarray}{\end{array}\end{equation}}

\newcommand{\fig}[3]{\begin{figure}[h]\begin{center}\includegraphics[#1]{#2}\end{center}\caption{#3}\label{fig:#2}\end{figure}}

\newcommand{\deriv}[1]{\frac{\partial}{\partial #1}}

\newcommand{\Gam}{\Gamma}
\newcommand{\Ups}{\Upsilon}

\newcommand{\eps}{\epsilon}
\newcommand{\De}{\Delta}

\newcommand{\nl}{\mathrm{nl}}
\newcommand{\oo}{\mathrm{out}}

\newcommand{\hU}{\hat{U}}
\newcommand{\hV}{\hat{V}}
\newcommand{\hP}{\hat{P}}

\newcommand{\mB}{\mathcal{B}}

\newcommand{\mO}{\mathcal{O}}

\newcommand{\mT}{\mathcal{T}}
\newcommand{\mS}{\mathcal{S}}

\newcommand{\NN}{\mathbb{N}}
\newcommand{\CC}{\mathbb{C}}

\newcommand{\PP}{P}
\newcommand{\QQ}{Q}

\usepackage[top=3cm, bottom=3cm, left=3cm, right=3cm]{geometry}

\begin{document}

\title[Analogue of the Harer-Zagier formula]{\bf An analogue of the Harer-Zagier formula for unicellular maps on general surfaces}


\author{Olivier Bernardi}
\address{O. Bernardi,  MIT, 77 Massachusetts Avenue, Cambridge MA 02139, USA.}
\thanks{Supported by the ANR project A3 and the European project ExploreMaps - ERC StG 208471.}

\begin{abstract}
A unicellular map is the embedding of a connected graph in a surface in such a way that the complement of the graph is simply connected. In a famous article, Harer and Zagier established a formula for the generating function of unicellular maps counted according to the number of vertices and edges. The keystone of their approach is a counting formula for unicellular maps on orientable surfaces with~$n$ edges, and with vertices colored using every color in~$[q]$ (adjacent vertices are authorized to have the same color). 
We give an analogue of this formula for general (locally orientable) surfaces.

Our approach is bijective and is inspired by Lass's proof of the Harer-Zagier formula. We first revisit Lass's proof and twist it into a bijection between unicellular maps on orientable surfaces with vertices colored using every color in~$[q]$, and maps with vertex set~$[q]$ on orientable surfaces \emph{with a marked spanning tree}. The bijection immediately implies Harer-Zagier's formula and a formula by Jackson concerning bipartite unicellular maps. It also shed a new light on constructions by Goulden and Nica, Schaeffer and Vassilieva, and Morales and Vassilieva. We then extend the bijection to general surfaces and obtain a correspondence between unicellular maps on general surfaces with vertices colored using every color in~$[q]$,  and maps on orientable surfaces with vertex set~$[q]$ \emph{with a marked planar submap}. This correspondence gives an analogue of the Harer-Zagier formula for general surfaces. We also show that this formula implies a recursion formula due to Ledoux for the numbers of unicellular maps with given numbers of vertices and edges.
\end{abstract}

\maketitle

\section{Introduction}
A \emph{map} is a cellular embedding of a connected graph on a surface considered up to homeomorphism (see Section~\ref{sec:orientable} for definitions).  A \emph{planar map} is a map on the sphere. A map is said \emph{unicellular} if it has a single face. For instance, the planar unicellular maps are the plane trees. 
A map is said \emph{orientable} if the underlying surface is orientable.\\

In~\cite{Harer-Zagier} Harer and Zagier considered the problem of enumerating orientable unicellular maps according to the number of edges and vertices (or, equivalently by Euler formula, according to the number of edges and the genus). The keystone of their approach is an enumerative formula for unicellular maps with \emph{colored} vertices (these colorings are not necessarily proper, that is, two adjacent vertices can have the same color): they proved that the number of rooted unicellular maps on orientable surfaces having $n$ edges and vertices colored using \emph{every} color in $[q]:=\{1,\ldots,q\}$ is 
\begin{equation}\label{eq:Harer-Zagier-every-color}
T_{n}(q)=2^{q-1}{n \choose q-1}(2n-1)!!,
\end{equation} 
where $(2n-1)!!=(2n-1)(2n-3)\cdots 1$.
 
From~\eqref{eq:Harer-Zagier-every-color}, it follows that the numbers $\eps_v(n)$ of unicellular maps on orientable surfaces with $n$ edges and $v$ vertices  satisfy
\begin{equation}\label{eq:Harer-Zagier-some-colors}
\sum_{v=1}^{n+1}\eps_v(n)N^v=\sum_{q=1}^{n+1} {N \choose q} 2^{q-1}{n \choose q-1}(2n-1)!! .
\end{equation}
since both sides of the equation represent the number of unicellular maps with vertices colored using \emph{some} of the colors in $[N]$ (since maps with $n$ edges have at most $n+1$ vertices, and the index $q$ on the right-hand-side corresponds to the number of colors really used). From this equation, Harer and Zagier obtained a recurrence relation for the numbers  $\eps_v(n)$ of orientable unicellular maps with $n$ edges and $v$ vertices:
\begin{equation}\nonumber 
(n+1)\eps_v(n)=(4n-2)\eps_{v-1}(n-1)+(n-1)(2n-1)(2n-3)\eps_{v}(n-2).
\end{equation}
The proof of~\eqref{eq:Harer-Zagier-every-color} in~\cite{Harer-Zagier} used a matrix integral argument (see~\cite{Lando-zvonkin} for an accessible presentation). A combinatorial proof was later given by Lass~\cite{Lass:Harer-Zagier}, and subsequently a bijective proof was given by Goulden and Nica~\cite{Goulden:Harer-Zagier}.\\

The goal of this paper is to give an analogue of Equation~\eqref{eq:Harer-Zagier-every-color} for unicellular maps on \emph{general} (i.e. locally orientable) surfaces, and to describe the bijections hiding behind it (see Equation~\eqref{eq:non-orientable-every-color} below). In order to do so, we first simplify slightly the proof of Lass for the orientable case, and then show how to extend it in order to deal with the general case.\\

In Section~\ref{sec:orientable}, we revisit Lass's proof of~\eqref{eq:Harer-Zagier-every-color} in order to obtain a bijection $\Phi$ between orientable unicellular maps colored using the every colors in $[q]$ and orientable \emph{tree-rooted maps} (maps with a marked spanning tree) with vertex set $[q]$. 
This bijective twist given to Lass's proof turns out to simplify certain calculations because tree-rooted maps are easily seen to be counted by the right-hand-side of~\eqref{eq:Harer-Zagier-every-color}. As mentioned above, a bijection was already exhibited by Goulden and Nica for unicellular maps on orientable surfaces in~\cite{Goulden:Harer-Zagier}. However, this bijection is presented in terms of the permutations (in particular the image of the bijection is not expressed in terms of maps) which makes its definition and analysis more delicate. By contrast, it is immediate to see that the bijection $\Phi$ satisfies the following property: 
\begin{enumerate}
\item[($\star$)] for any unicellular map $U$ colored using every colors in $[q]$ and for any colors $i,j\in[q]$, the number of edges of $U$ with endpoints of color $i,j$ is equal to the number of edges between the vertices $i,j$ in the tree-rooted map $\Phi(U)$. 
\end{enumerate}
This property leads in particular to a bijective proof of the following formula by Jackson~\cite{Jackson:kpartite-unicellular-maps} (closely related to another formula found independently by Adrianov~\cite{Adrianov:bicolored-unicellular}) for the number  $T_n(p,q)$ of orientable \emph{bipartite} unicellular maps with the sets of vertices colored with colors $\{1,\ldots,p\}$ and colors $\{p+1,\ldots,p+q\}$ respectively forming the bipartition (the root-vertex being colored in $[p]$): 
\begin{equation}\label{eq:Jackson-bipartite-every-color}
T_n(p,q)=n!{n-1 \choose p-1,q-1,n-p-q+1}.
\end{equation}
A bijective proof of~\eqref{eq:Jackson-bipartite-every-color} was already given by Schaeffer and Vassilieva~\cite{Schaeffer:bij-Jackson-formula}. Again, the presentation there is in terms of permutations. In~\cite{Morales-Vassilieva:unicellular-bicolored-degree}, Morales and Vassilieva analyzed the bijection in~\cite{Schaeffer:bij-Jackson-formula} in order to enumerate bipartite unicellular maps colored in $[q]$ according to the sum of degrees of the vertices of each color. These refined enumerative results can be obtained using the property ($\star$). Thus, in the orientable case, the bijection $\Phi$ gives a unified ways of obtaining the bijective results in~\cite{Goulden:Harer-Zagier,Schaeffer:bij-Jackson-formula,Morales-Vassilieva:unicellular-bicolored-degree}.\\

In Section~\ref{sec:general}, we turn to general surfaces. We extend the bijection $\Phi$ into a correspondence between unicellular maps colored using every color in $[q]$ and orientable \emph{planar-rooted maps} (maps with a marked spanning connected planar submap\footnote{As recalled in Section~\ref{sec:orientable}, orientable maps can be defined in terms of \emph{rotation systems}. A submap of an oriented map $M$ is said \emph{planar} if its rotation system inherited from $M$ defines a planar map.}) with vertex set $[q]$. One of the new ingredients used to prove the correspondence $\Phi$ is a bijection between planar maps and certain ``decorated'' plane trees (which is closely related to a bijection by Schaeffer~\cite{Schaeffer:eulerian} for Eulerian planar maps). Proofs concerning this bijection are given in Section~\ref{sec:proof-planar}.
The correspondence $\Phi$ between colored unicellular maps and planar-rooted maps is not one-to-one, but rather $(2^{r-1})$-to-$(r!)$, where $r$ is the number of faces of the planar submap. This yields the following formula for the number $U_n(q)$ of rooted unicellular maps colored using every color in $[q]$:
\begin{equation}\label{eq:non-orientable-every-color}
U_{n}(q)=\sum_{r=1}^{n-q+2}\frac{q!r!}{2^{r-1}}P_{q,r}\,{2n \choose 2q+2r-4}(2n-2q-2r+1)!!,
\end{equation}
where $P_{q,r}$ is the number of  (unlabeled) rooted planar maps with $q$ vertices and $r$ faces.
More refined enumerative results (in the spirit of~\cite{Morales-Vassilieva:unicellular-bicolored-degree}) can be obtained from the correspondence $\Phi$ because this bijection satisfies:
\begin{enumerate}
\item[($\star\star$)] if a unicellular map $U$  and a planar-rooted map $M$ are in correspondence by $\Phi$, then for any color $i$ the sum of the degrees of the vertices colored $i$ in $U$ is equal to the degree of the vertex $i$ in $M$.
 \end{enumerate}


Let us situate our work among some other enumerative results for unicellular maps on general surfaces. 
In~\cite{Goulden:maps-nonorientable+unicellular} Goulden and Jackson (using representation theory techniques) gave a way of expressing the generating functions of maps on general surfaces in terms of matrix integrals over the Gaussian Orthogonal Ensemble. This expression allowed them to prove a formula for the number  $\hU_{n}(N)$ of unicellular maps on general surfaces with vertices colored using \emph{some} of the colors in $[N]$:
\begin{equation}\label{eq:non-orientable-some-colors}
\hU_{n}(N)=n!\sum_{k=0}^n 2^{2n-k}\sum_{r=0}^n {n-\frac12 \choose n-r}{k+r-1 \choose k}{\frac{N-1}{2} \choose r}+ (2n-1)!!\sum_{q=1}^{N-1}2^{q-1}{N-1 \choose q}{n \choose q-1},
\end{equation}
where for any real number $x$ and any integer $k$ the notation ${x \choose k}$ stands for $x(x-1)\cdots (x-k+1)/k!$. Given that $\hU_{n}(N)=\sum_{q=1}^N{n \choose q} U_{n}(q)$, it is natural to try to prove~\eqref{eq:non-orientable-every-color} starting from~\eqref{eq:non-orientable-some-colors} or the converse, but we have failed to do so.
In~\cite{Ledoux:recursion-unicellular}, Ledoux showed that the matrix integral considered by Goulden and Jackson satisfies a differential equation, implying a recursion formula for the numbers $\eta_v(n)$ of unicellular maps with $n$ edges and $v$ vertices:
\begin{equation}\label{eq:rec-ledoux}
\begin{array}{rcl}
(n+1)\,\eta_v(n)&=&(4n-1)\left(2\,\eta_{v-1}(n-1)-\eta_v(n-1)\right)\\
&&+(2n-3)\left((10n^2-9n)\,\eta_v(n-2)+8\,\eta_{v-1}(n-2)-8\,\eta_{v-2}(n-2)\right)\\
&&+5(2n-3)(2n-4)(2n-5)\left(\eta_v(n-3)-2\,\eta_{v-1}(n-3)\right)\\
&&-2(2n-3)(2n-4)(2n-5)(2n-6)(2n-7)\,\eta_v(n-4),
\end{array}
\end{equation}
valid for all $n\geq 2$. We explain in Section~\ref{sec:recurrence} how to prove this relation starting from~\eqref{eq:non-orientable-every-color}.

We mention lastly that a different bijective approach for orientable unicellular maps was introduced by Chapuy~\cite{Chapuy:unicellular} (giving~\eqref{eq:Harer-Zagier-some-colors} and some new results). In~\cite{Bernardi-Chapuy:unicellular-non-orientable}, this approach was partially adapted to general surfaces, giving an exact formula for the number of unicellular \emph{precubic maps} (maps with vertices of degree either 1 or 3). However, no exact formula for non-precubic unicellular maps on general surfaces has been derived so far by this method.\\

 \section{Unicellular maps and Eulerian tours}\label{sec:defs}
In this section we set our definitions and extend the correspondence used in~\cite{Lass:Harer-Zagier} between colored unicellular maps and certain Eulerian tours on graphs.

Our \emph{graphs} can have loops and multiple edges. A graph $G$ with $n$ edges is said \emph{edge-labelled} if its edges are labelled with distinct numbers in $[n]$. A \emph{map} is a cellular embedding of a connected graph in a (2-dimensional, smooth, compact, boundaryless) surface considered up to homeomorphism. By \emph{cellular} we mean that the \emph{faces} (connected components of the complement of the graph) are simply connected.  Thus, maps are the \emph{gluings of polygons} obtained by identifying the edges of the polygons two by two, considered up to homeomorphism. In particular, the \emph{unicellular maps} (maps with a single face) with $n$ edges are the gluings obtained from a $2n$-gon (polygon with $2n$ edges); see Figure~\ref{fig:unicellular-oriented}.
A polygon is \emph{rooted} by orienting one of its edges and calling it the \emph{root-edge}. A \emph{rooted-map} is obtained by gluing polygons one of which is rooted; this distinguish one of the edges as the \emph{root-edge} and specifies a direction and a side for the root. For instance, the root-edge, its side and direction are indicated by an arrow labelled 1 in Figure~\ref{fig:unicellular-oriented}. The \emph{root-vertex} is the origin of the root-edge. 

\fig{width=10cm}{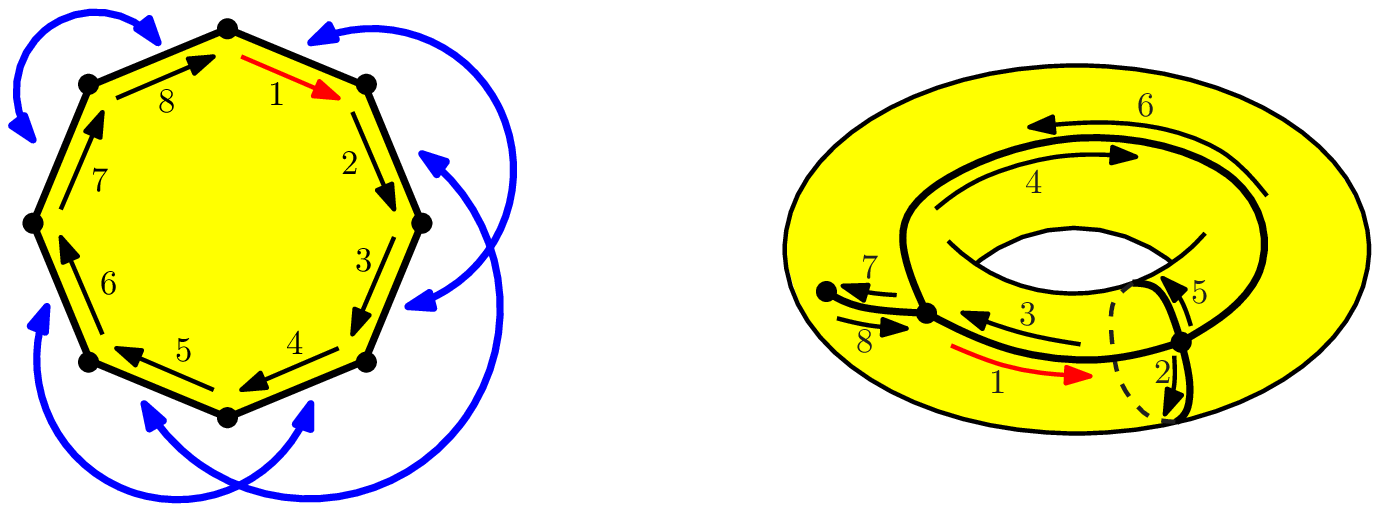}{An orientable rooted unicellular map seen as a gluing of a polygon (left), and seen as an embedded graph (right). The root is indicated by the arrow labelled 1.}


Given two edges of a polygon, there are two ways of gluing them together: we call \emph{orientable gluing} the identification of the edges giving a topological cylinder, and \emph{non-orientable gluing} the identification giving a M\"obius band. 
It is easy to see that a unicellular map obtained by gluing the edges of a polygon in pairs is orientable if and only if each of the gluing is orientable.
Hence there are $2^n(2n-1)!!$ rooted unicellular maps with $n$ edges and $(2n-1)!!$ orientable ones (since $(2n-1)!!$ is the number of ways of partitioning a set of $2n$ elements in $n$ non-ordered pairs). Observe that if the edges of the rooted $2n$-gon are oriented into a directed cycle, then the orientable gluings are the one for which any pair of edges glued together have opposite orientations along the gluing. 

We call $q$\emph{-colored} a map in which the vertices are colored using \emph{every color} in $[q]=\{1,\ldots,q\}$ (adjacent vertices can have the same color). A key observation made in~\cite{Lass:Harer-Zagier} is that making the tour of the face of a $q$-colored unicellular map starting from the root induces a tour on a graph having vertex set $[q]$. We now make this statement precise. Les $G$ be an undirected edge-labelled graph. We call \emph{bi-Eulerian} tour of $G$ a directed path starting and ending at the same vertex called the \emph{origin} and using every edge exactly twice. If $G$ has loops, the convention is that tours can either use the loop twice in the same direction, or in the two opposite directions (so that there are 2 different bi-Eulerian tours on the graph consisting of a single loop, and more generally, there are $(2n)!$ bi-Eulerian tours on an edge-labelled graph consisting of $n$ loops). 

\begin{lemma}\label{lem:uni=Eulerian}
There is a bijection $\Xi$ between rooted edge-labelled $q$-colored unicellular maps with $n$ edges and the set of pairs $(G,B)$, where $G$ is an edge-labelled connected graph with $n$ edges and vertex set $[q]$, and $B$ is a bi-Eulerian tour of $G$. Moreover, if $(G,B)=\Xi(U)$ then 
\begin{itemize} 
\item the unicellular map $U$ is orientable if and only if the tour $B$ never uses an edge twice in the same direction,
\item for any colors $i,j$ in $[q]$ the number of edges in $U$ having endpoints of colors $i,j$ is equal to the number of edges in $G$ having endpoints $i,j$.
\end{itemize} 
\end{lemma}
 
Lemma~\ref{lem:uni=Eulerian} is illustrated by Figure~\ref{fig:unicolor}.

\fig{width=12cm}{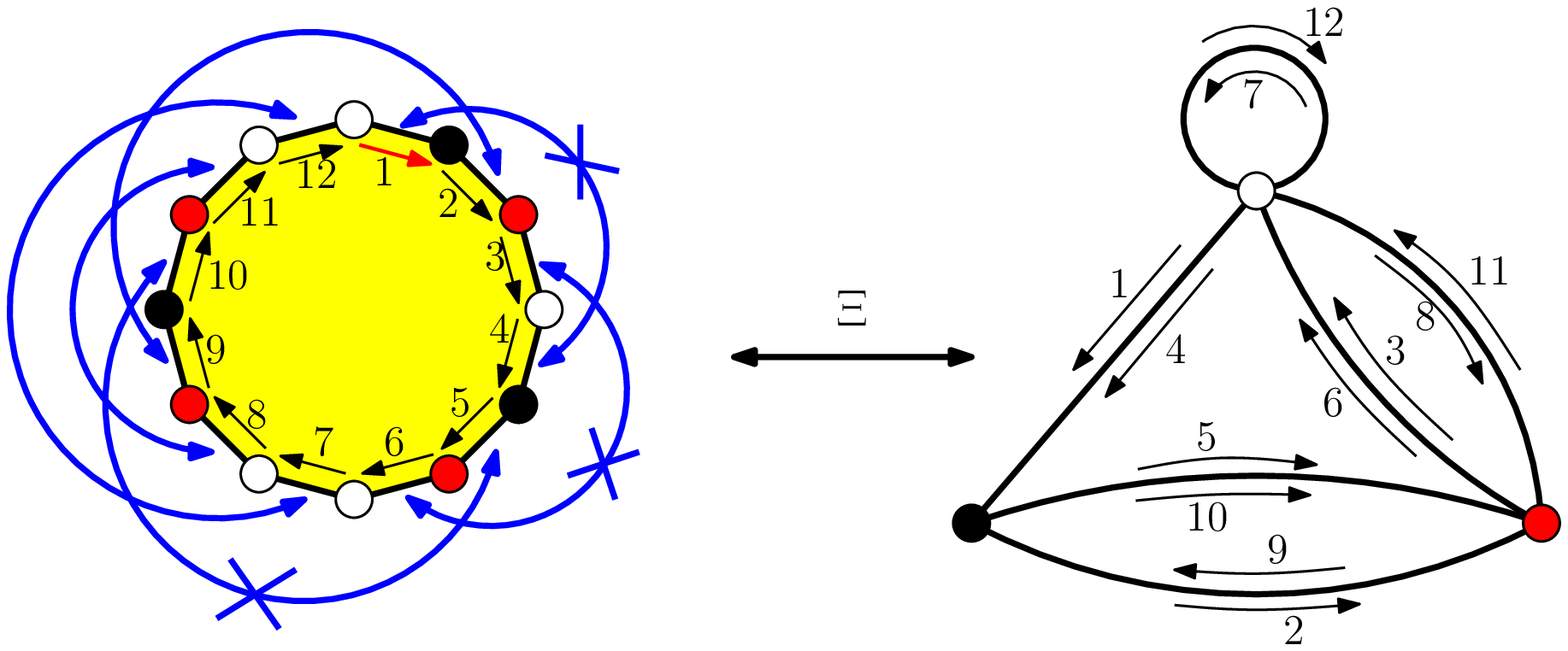}{A gluing of a 3-colored rooted 12-gon (giving a non-orientable 3-colored unicellular map), and the corresponding Eulerian tour. The non-orientable gluings are marked with a cross.}

\begin{proof} 
A $q$-colored unicellular map $U$ is obtained by coloring the vertices of a rooted $2n$-gon with every color in $[q]$ and then choosing a gluing of the edges respecting the colors. We consider the orientation of the $2n$-gon into a directed cycle respecting the orientation of the root-edge. Choosing a coloring of the vertices of the $2n$-gon using every colors in $[q]$ is the same as choosing a connected directed graph $\vec{G}$ with vertex set $[q]$ together with an Eulerian tour $E$ of $\vec{G}$ (a directed path starting and ending at the same vertex and using every edge twice). Indeed, the $k$th edge of the Eulerian tour identifies with the $k$th edge around the $2n$-gon (starting from the root-edge). 

Now, given a colored $2n$-gon and the associated pair $(\vec{G},E)$, the gluings of the edges of the polygon respecting the colors are clearly in bijection with the partitions $\pi$ of the $2n$ arcs of $\vec{G}$ into pairs of arcs having the same endpoints, with the convention that two loops incident to the same vertex can be paired in two ways (either paired ``in opposite direction'' or ``in the same direction''). Also, labelling the edges of the unicellular map obtained is the same as labelling the pairs in $\pi$. 
Lastly, the triples $(\vec{G},E,\pi)$ clearly identify with the pairs $(G,B)$, where $G$ is an edge-labelled connected graph with $n$ edges and vertex set $[q]$, and $B$ is a bi-Eulerian tour of $G$. The bijection established has all the claimed properties.
\end{proof}

\medskip

\section{Orientable unicellular maps}\label{sec:orientable}
In this section, we twist Lass's proof~\cite{Lass:Harer-Zagier} of~\eqref{eq:Harer-Zagier-every-color} into a bijection between colored unicellular maps and tree-rooted maps. The bijection $\Phi$ obtained can be seen as the (unified) ``map version'' of the bijections presented in terms of permutations in~\cite{Goulden:Harer-Zagier} and~\cite{Schaeffer:bij-Jackson-formula} to prove respectively~\eqref{eq:Harer-Zagier-every-color} and~\eqref{eq:Jackson-bipartite-every-color}.

We first recall the so-called \emph{BEST Theorem}. An \emph{Eulerian tour} of a directed graph $G$ is a directed path starting and ending at the same vertex called the \emph{origin} and using each arc exactly once.  A spanning tree of $G$ is \emph{directed toward} a vertex $v_0$ if the path between any vertex vertex $v$ and $v_0$ in $T$ is directed toward $v_0$.
 
\begin{lemma}[BEST Theorem]  
Let $G$ be an edge-labelled directed graph such that any vertex is incident to as many ingoing and outgoing arcs, and let $v_0$ be a vertex.
Then, the set of Eulerian tours of $G$ with origin $v_0$ is in bijection with the set of pairs $(T,\tau)$, where $T$ is a spanning tree directed toward $v_0$, and $\tau$ is a choice for each vertex $v$ of a total order of the outgoing arcs not in $T$. 
\end{lemma}

The proof of the BEST Theorem can be found in~\cite{Stanley:volume2}\footnote{To an Eulerian tour with origin $v_0$, the bijection  associates the pair $(T,\tau)$, where $T$ is the spanning tree oriented toward $v_0$ made of the last arcs used to leave each of the vertices $v\neq v_0$, and the total order $\tau$ at a vertex $v$ records the order in which the other arcs leaving $v$ are used during the tour.}. Following the method of Lass we will now use the BEST Theorem to count bi-Eulerian tours. But we will shortcut this counting by considering \emph{rotation systems} of graphs.
Let $G$ be a graph. A \emph{half-edge} is an incidence between an edge and a vertex (one can think of half-edges as obtained from cutting an edge at a midpoint); by convention a loop corresponds to two half-edges both incident to the same vertex.  A \emph{rotation system} of $G$ is a choice for each vertex $v$ of a cyclic ordering of the half-edges incident to $v$. A \emph{rooted rotation system} with \emph{root-vertex} $v_0$ is a choice for each vertex $v\neq v_0$ of a cyclic order of the  half-edges incident to $v$ and the choice of a total order of the half-edges incident to $v_0$. We now recall how to associate a rooted rotation system to a rooted orientable map $M$. Given a rooted map $M$ of underlying graph $G$ on an orientable surface $\mS$, one can define the ``positive orientation'' of the surface $\mS$ by requiring that the marked side of the root-edge precedes (the outgoing half-edge of) the root-edge in counterclockwise order around the root-vertex. Then one defines the \emph{rotation system} $\rho(M)$ of $G$ as follows: around any non-root vertex $v$ the rotation system $\rho(M)$ is the cyclic order of the half-edges in counterclockwise direction around $v$, and around the root-vertex $v_0$ the rotation system  $\rho(M)$ is the total order obtained by breaking the cyclic counterclockwise order in order to make the outgoing half-edge of the root-edge be the smallest element. The proof of the following classical result can be found in~\cite{Mohar:graphs-on-surfaces}.

\begin{lemma}[Embedding Theorem]  
For any connected graph $G$, the mapping $\rho$ is a bijection between the rooted orientable maps of underlying graph $G$ and the rooted rotation systems of~$G$. 
\end{lemma}

We are now ready to state the relation between bi-Eulerian tours and \emph{tree-rooted maps} (rooted maps with a marked spanning tree).

\begin{lemma}\label{lem:Eulerian=embedding}
Let $G$ be a connected edge-labelled undirected graph.  The set of bi-Eulerian tours of $G$ that never take an edge twice in the same direction are in bijection with pairs $(T,\sigma)$ made of a spanning tree $T$ and a rooted rotation system $\sigma$. Consequently, this set of bi-Eulerian tours is in bijection with the set of tree-rooted maps on orientable surfaces having underlying graph~$G$.  
\end{lemma}

\begin{proof} 
Let $v_0$ be a vertex of $G$. We will establish a bijection between the set $\mB_0$ of bi-Eulerian tours of $G$ of origin $v_0$ that never take an edge twice in the same direction and the set $\mT_0$ of pairs $(T,\sigma)$ made of a spanning tree $T$ and a rooted rotation system $\sigma$ with root-vertex $v_0$.

Let $\vec{G}$ be the directed graph obtained from $G$ by replacing each edge $e$ by two arcs in opposite directions taking the same label as $e$. Observe that the arcs of $\vec{G}$ are all distinguishable except for pairs of arcs corresponding to a loop of $G$. The set $\mB_0$ of bi-Eulerian tours of $G$ is clearly in bijection with the set of Eulerian tours of $\vec{G}$ having origin $v_0$. Hence, by the BEST Theorem, the set $\mB_0$ is in bijection with the set $\vec{\mT}_0$ of pairs $(T_0,\tau)$,  where $T_0$ is a spanning tree of $\vec{G}$ directed toward $v_0$ and $\tau$ is a choice of a total ordering of the outgoing arcs of $\vec{G}$ not in $T_0$ at each vertex. Again here if two arcs comes from a loop of $G$, the total order $\tau$ does not distinguish between these two arcs.

It only remains to show that the set $\vec{\mT}_0$ is in bijection with the set $\mT_0$. First of all, the spanning trees of $\vec{G}$ directed toward $v_0$ are in natural bijection with the spanning trees of $G$. Moreover, the arcs of $\vec{G}$ and the half-edges of $G$ are in natural correspondence (with arcs of $\vec{G}$ indistinguishable if and only if they come from the same loop, and half-edges of $G$ indistinguishable if and only if they come from the same loop).
Therefore, for the vertex~$v_0$, the total orderings of the outgoing arcs of $\vec{G}$ incident to $v_0$ correspond bijectively to the total orderings of the half-edges of $G$ incident to $v_0$. Similarly,  for any vertex $v\neq v_0$ and any spanning tree $T_0$ directed toward $v_0$, the total orderings of the outgoing arcs of $\vec{G}$ not in $T_0$ (that is, the total orderings of all the outgoing arcs except the one leading $v$ to its parent in $T_0$) corresponds bijectively to the cyclic orderings of the half-edges of $G$ incident to $v$. This gives a bijection between $\vec{\mT}_0$ and $\mT_0$.
\end{proof}

We now summarize the consequences of Lemmas~\ref{lem:uni=Eulerian} and~\ref{lem:Eulerian=embedding}. 

\begin{theorem} \label{thm:orientable}
There is a bijection $\Phi$ between $q$-colored rooted unicellular maps with $n$ edges on orientable surfaces and tree-rooted maps with $n$ edges and vertex set $[q]$ on orientable surfaces. 
Moreover for all $i,j\in[q]$ the number of edges with endpoints of colors $i,j$ in a $q$-colored unicellular map $U$ is equal to the number of edges between vertices $i$ and $j$ in the tree-rooted map $\Phi(U)$.
\end{theorem}

\fig{width=\linewidth}{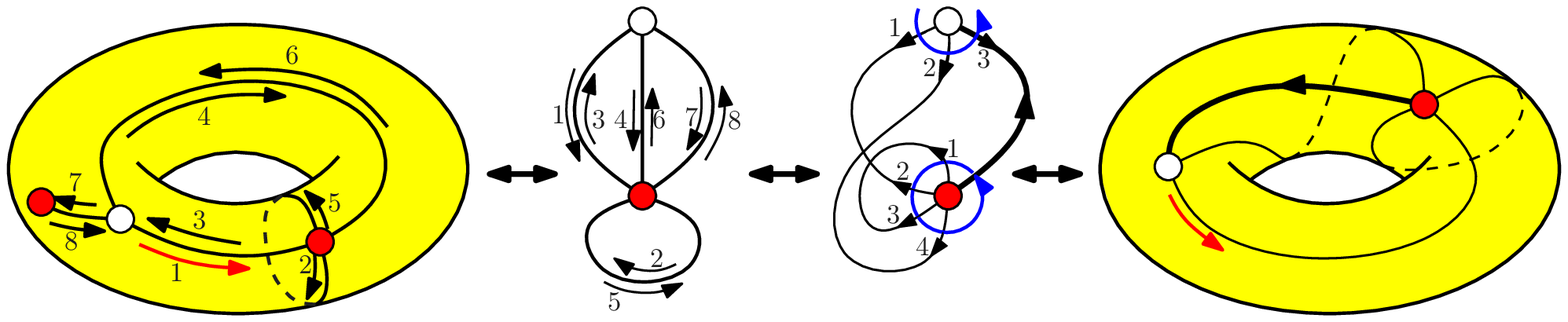}{Bijection between orientable colored unicellular maps and tree-rooted maps (spanning trees are drawn in thick lines).}

\begin{proof} Lemmas~\ref{lem:uni=Eulerian} and~\ref{lem:Eulerian=embedding} imply a bijection having the claimed properties between \emph{edge-labelled} maps. However, the edges of a non-labelled rooted maps are all distinguishable. Hence each non-labelled rooted maps with $n$ edges correspond to $n!$ different edge-labelled rooted maps.
\end{proof}

Theorem~\ref{thm:orientable} is illustrated in Figure~\ref{fig:bijection-orientable}. Before closing this section we explain how to obtain~\eqref{eq:Harer-Zagier-every-color} and~\eqref{eq:Jackson-bipartite-every-color} from Theorem~\ref{thm:orientable}. We start with the proof of~\eqref{eq:Harer-Zagier-every-color}. By Theorem~\ref{thm:orientable}, $T_n(q)$ is the number of tree-rooted maps with $n$ edges and vertex set $[q]$ on orientable surfaces. One can think of a tree-rooted maps in terms of its rotation system: one can first choose the tree together with its rooted rotation system (that is, a rooted plane tree) and then add the other edges. 
It is well known that the rooted plane trees with $q$ unlabeled vertices are counted by the Catalan number $\frac{(2q-2)!}{q!(q-1)!}$, hence the number of  trees with vertex set $[q]$ is $\frac{(2q-2)!}{(q-1)!}$. In order to add the $n-q+1$ edges not in the tree, we need first to choose where to insert the $2n-2q+2$ half-edges around the tree (equivalently, how to complete the rooted rotation system of the tree with new half-edges). A rooted plane tree with 
$q$ vertices has $2q-1$ positions where a half-edge can be inserted, hence there are   ${2n \choose 2q-2}$ ways to add the $2n-2q+2$ additional half-edges around the tree. Lastly  there are $(2n -2q+1)!!$ ways of pairing these half edges.  This gives~\eqref{eq:Harer-Zagier-every-color}: 
$$T_n(q)=\frac{(2q-2)!}{(q-1)!}\, {2n \choose 2q-2} \, (2n -2q+1)!!= 2^{q-1}{n \choose q-1}(2n-1)!!.$$

We prove~\eqref{eq:Jackson-bipartite-every-color} in a similar fashion. By Theorem~\ref{thm:orientable}, $T_n(p,q)$  is the number of (bipartite) tree-rooted maps on orientable surfaces with $n$ edges and vertex set $[p+q]$ such that any edge joins a vertex in $A=[p]$ to a vertex in $B=[p+1,\ldots,p+q]$, and the root-vertex is in $A$. It is well known that the bicolored trees with $p$ black unlabeled vertices, and $q$ white unlabeled vertices (with a black root-vertex and every edge joining a black vertex to a white vertex) are counted by the Narayama number $\frac{(p+q-1)!(p+q-2)!}{p!(p-1)!q!(q-1)!}$. Hence, there are $\frac{(p+q-1)!(p+q-2)!}{(p-1)!(q-1)!}$ ways to choose a vertex-labelled tree such that any edge joins a vertex in $A$ to a vertex in $B$, and the root-vertex is in $A$. The $n-p-q+1$ other edges have one half-edge incident to a vertex in $A$ and one half-edge incident to a vertex in  $B$. There are $p+q$ positions to insert a half-edge around vertices in $A$, and $p+q-1$  positions to insert a half-edge around vertices in $B$. Hence,  there are ${n\choose p+q}{n-1\choose p+q-1}$ ways to add the  $2n-2q+2$ half-edges around the tree. Lastly, there are $(n-p-q+1)!$ ways of pairing these half-edges. This gives~\eqref{eq:Jackson-bipartite-every-color}.\\


\section{General unicellular maps}\label{sec:general}
In this section we extend the bijection $\Phi$ to  unicellular maps on general surfaces. We start with some definitions. A \emph{balanced orientation} of an undirected graph $G$ is an orientation of a (possibly empty) subset of edges of $G$ such that for any vertex $v$ the numbers of ingoing and outgoing half-edges incident to $v$ are equal. If $G$ has a vertex $v_0$ distinguished as the \emph{root-vertex}, then we say that a balanced orientation $\mO$ and a spanning tree $T$ are \emph{compatible} if any oriented edge in $T$ has an orientation that coincides with the orientation of $T_0$ toward $v_0$. The \emph{external weight} of the pair $(\mO,T)$ is the number of oriented edges not in $T$. We call \emph{compatibly-oriented tree-rooted map}, a rooted map with a marked spanning tree and a compatible balanced orientation. 
For an edge-labelled graph $G$, we denote by $\mT(G)$ the set of edge-labelled compatibly-oriented tree-rooted maps with underlying graph $G$, and we denote by $\mB(G)$ the set of bi-Eulerian tours of~$G$. 



\begin{lemma}\label{lem:Eulerian=embedding2}
Let $G$ be a connected edge-labelled undirected graph. There exists a surjective mapping $\Theta$ between $\mT(G)$ and $\mB(G)$ satisfying: for any bi-Eulerian tour $B$ in $\mB(G)$ there is an  integer $w=w(B)\geq 0$ such that the preimage $\Theta^{-1}(B)$ consists of $2^{w}$ compatibly-oriented tree-rooted maps of external weight $w$. Moreover $w(B)=0$ if and only if the bi-Eulerian tour $B$ never takes an edge twice in the same direction
\end{lemma}

\begin{proof} 
Observe that to any bi-Eulerian tour $B$ of $G$, one can associate a balanced orientation $\mO(B)$ of $G$ by orienting the edges of $G$ which are taken twice in the same direction during the tour $B$ according to this direction. We now fix a vertex $v_0$ and a balanced orientation $\mO$ and show that there is  a mapping $\Theta$ with the desired properties between set $\mT_0(\mO)\subset \mT(G)$ of compatibly-oriented tree-rooted maps with root-vertex $v_0$ and balanced orientation $\mO$, and the set  $\mB_0(\mO)\subset \mB(G)$ of bi-Eulerian tours with origin $v_0$ such that $\mO(B)=\mO$.

Let $\vec{G}_\mO$ be the directed graph obtained by replacing each edge $e$ of $G$ by two arcs with the same label as $e$: the non-oriented edges of $\mO$ are replaced by two arcs in opposite directions, while the oriented edges of $\mO$ are replaced by two arcs in the direction of $e$. In $\vec{G}_\mO$ the two arcs that come from a loop of $G$ or from an oriented edge of $\mO$ are indistinguishable. Clearly the set $\mB_0(\mO)$ of bi-Eulerian tours of $G$ is in bijection with the set of Eulerian tours of origin $v_0$ on $\vec{G}_\mO$.  Hence, by the BEST Theorem, the set  $\mB_0(\mO)$  is in bijection with the set $\vec{\mT}_0(\mO)$ of pairs $(T_0,\tau)$ made of a spanning tree $T_0$ of $\vec{G}_\mO$ directed toward $v_0$ and a choice $\tau$ for each vertex $v$ of a total order on the outgoing arcs not in $T_0$ incident to $v$. Again here if two arcs comes from a loop of $G$ or from a directed edge of $G$, the spanning tree $T_0$ and total orders $\tau$ do not distinguish between these two arcs.

It remains to establish the relation between $\mT_0(\mO)$ and $\vec{\mT}_0(\mO)$. Observe first that the spanning trees of  $\vec{G}_\mO$ directed toward $v_0$ are in bijection with the spanning trees of $G$ that are compatible with $\mO$. Let $T_0$ be a spanning tree of $\vec{G}_\mO$ directed toward $v_0$ and let $w$ be the external weight of the associated spanning tree $T$ of $G$. We only need to show that the number of rooted rotation systems of $G$ with root-vertex $v_0$ is $2^w$ times the number of choices $\tau$ for each vertex $v$ of $G$ of a total order on the outgoing arcs not in $T_0$. We first consider a vertex $v\neq v_0$.  The half-edges of $G$ incident to $v$ are all distinguishable, except for pairs of half-edges belonging to a non-oriented loop. Hence, there are $(\deg(v)-1)!/2^{\nl(v)}$ different cyclic orderings of the half-edges incident to $v$, where $\deg(v)$ is the degree of $v$ and $\nl(v)$ is the number of incident non-oriented loops. On the other hand, there are $\deg(v)-1$ outgoing arcs incident to $v$ not in $T_0$, and these arcs are all distinguishable except for pairs of half-edges belonging to a loop or to an oriented edge of $\mO$ not in $T_0$. Hence, there are $(\deg(v)-1)!/2^{\nl(v)+\oo(v)}$ different total orderings of the outgoing arcs not in $T_0$ incident to $v$, where $\oo(v)$ is the number of edges of $\mO$ not in $T_0$ oriented away from $v$. A similar reasoning applies for the root-vertex  $v_0$, and shows that there are $2^{\sum_{v} \oo(v)}$ times more rooted rotation systems of $G$ with root-vertex $v_0$ than of choices $\tau$ for each vertex of a total order on the outgoing arcs not in $T_0$. Moreover, the sum   $\sum_{v}\oo(v)$ is equal to the external weight $w$. Lastly, the external weight is 0 if and only if there is no oriented edge in $\mO$ (equivalently the tours in $\mB_0(\mO)$ never take an edge twice in the same direction).
\end{proof}

Putting Lemmas~\ref{lem:uni=Eulerian} and~\ref{lem:Eulerian=embedding2} together (and getting rid of the edge-labellings as in the previous section) one gets: 

\begin{theorem} \label{thm:non-orientable}
Let $q$ be a positive integer. There is a surjective mapping $\Ups$ from the set of compatibly-oriented tree-rooted maps with vertex set $[q]$ on orientable surfaces to the set of $q$-colored unicellular maps on general surfaces satisfying: for all $q$-colored unicellular map $U$,
\begin{itemize}
\item there is an integer $w=w(U)\geq 0$ such that the preimage $\Ups^{-1}(U)$ consists of $2^{w}$ compatibly-oriented tree-rooted maps of external weight $w$; moreover $w(U)=0$ if and only if  $U$ is  orientable,
\item for all $i,j\in[q]$ the number of edges with endpoints of colors $i,j$ in $U$ is equal to the number of edges with endpoints $i,j$ in any compatibly-oriented tree-rooted map in $\Ups^{-1}(U)$.
\end{itemize}
\end{theorem}

We now proceed to count compatibly-oriented tree-rooted maps according to their external weight. We first focus on the simpler structure of near-Eulerian trees. A \emph{near-Eulerian tree} is a rooted plane tree with a (possibly empty) subset of edges oriented toward the root-vertex, together with some dangling half-edges, called \emph{buds}, which are either ingoing or outgoing (at each vertex the rotation system includes both the half-edges in the tree and the buds) such that that at any vertex $v$ the numbers of ingoing and outgoing half-edges (including the buds) are equal.  An example of near-Eulerian tree is given in Figure~\ref{fig:closure}(left). It is easily seen that near-Eulerian trees have the same numbers of ingoing and outgoing buds; we call \emph{external weight} this number. 

Observe that near-Eulerian trees of external weight $w$ are the maps obtained from compatibly-oriented tree-rooted maps by deleting 
 all the non-oriented edges not in the spanning tree, and cutting in two halves the oriented edges not in the spanning tree. 
We will now exhibit a bijection  $\Gamma$ between near-Eulerian trees and  \emph{rooted plane maps}, that is, rooted planar maps (in the sphere) with a marked face\footnote{The bijection $\Gamma$ can be thought as an extension of a bijection due to Schaeffer \cite{Schaeffer:eulerian}  between \emph{Eulerian trees} (near-Eulerian trees in which all the edges of the tree are oriented) and rooted planar maps with vertices of even degree. However, there is a slight difference in the presentation, since the bijection of Schaeffer was actually mapping a subset of Eulerian trees (called \emph{balanced} Eulerian trees) to the set of rooted planar maps with vertices of even degree but no marked face.}.  It will be convenient to identify rooted plane maps with rooted maps embedded in the plane (the marked face become the infinite face of the plane).  If a near-Eulerian tree $T$ is embedded in the plane, an edge drawn from an outgoing bud to an ingoing bud of $T$ is said \emph{counterclockwise around $T$} if it has the infinite face on its right. There is a unique way to draw a set of non-crossing counterclockwise edges around $T$ joining each outgoing bud to an ingoing bud; see Figure~\ref{fig:closure}. We call this process the  \emph{closure} of $T$, and denote by $\Gam(T)$ the rooted plane map obtained.

\fig{width=12cm}{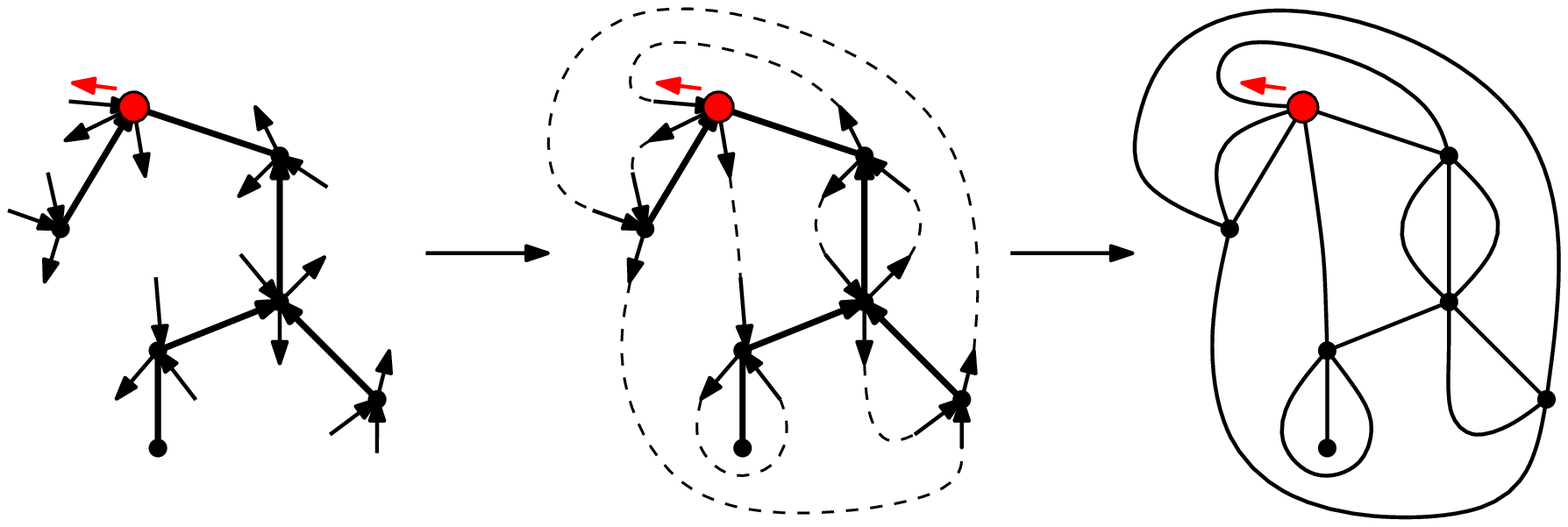}{The closure $\Gamma$: ingoing and outgoing buds are joined counterclockwise around $T$.}

\begin{theorem}\label{thm:Eulerian-bijection}
The closure $\Gam$ is a bijection between near-Eulerian trees and rooted plane maps.
\end{theorem}

The proof of Theorem~\ref{thm:Eulerian-bijection} is delayed to the next section. We now state its consequences for unicellular maps. A \emph{spanning submap} of a map is a subset of edges giving a connected subgraph containing every vertex. A spanning submap $S$ of an orientable map $M$ is said \emph{planar} if the rotation system of $S$ inherited from $M$ makes $S$ a planar map. A \emph{planar-rooted map} is a rooted map on an orientable surface with a marked spanning planar submap. A $r$\emph{-externally-labelled} planar-rooted map is a pair  $(P,\pi)$ where $P$ is a planar-rooted map whose marked submap has $r$ faces, and  $\pi$ is a permutation of $[r]$.

\begin{theorem}\label{thm:non-orientable-bijection} 
Let $q$ be a positive integer. There exists a surjective mapping $\Psi$ from the set of externally-labelled planar-rooted maps with vertex set $[q]$ on orientable surfaces, to the set of $q$-colored rooted unicellular maps on general surfaces satisfying: for all $q$-colored unicellular map $U$,
\begin{itemize}
\item there is an integer $w=w(U)\geq 0$ such that the preimage $\Psi^{-1}(U)$ consists of $2^{w}$ $(w+1)$-externally-labelled planar-rooted maps; moreover $w(U)=0$ if and only if $U$ is orientable,
\item for all $i\in[q]$ the sum of degrees of the vertices of $U$ colored $i$ is equal to the degree of the vertex $i$ in any planar-rooted map in $\Psi^{-1}(U)$.
\end{itemize}
\end{theorem}

Theorem~\ref{thm:non-orientable-bijection} is illustrated in Figure~\ref{fig:summary-bijection}. Observe that tree-rooted maps can be seen as 1-externally-labelled planar-rooted maps, and that the mapping $\Psi$ given in Theorem~\ref{thm:non-orientable-bijection} can be seen as the extension of the bijection $\Phi^{-1}$ given in Theorem~\ref{thm:orientable} for orientable unicellular maps.

\fig{width=\linewidth}{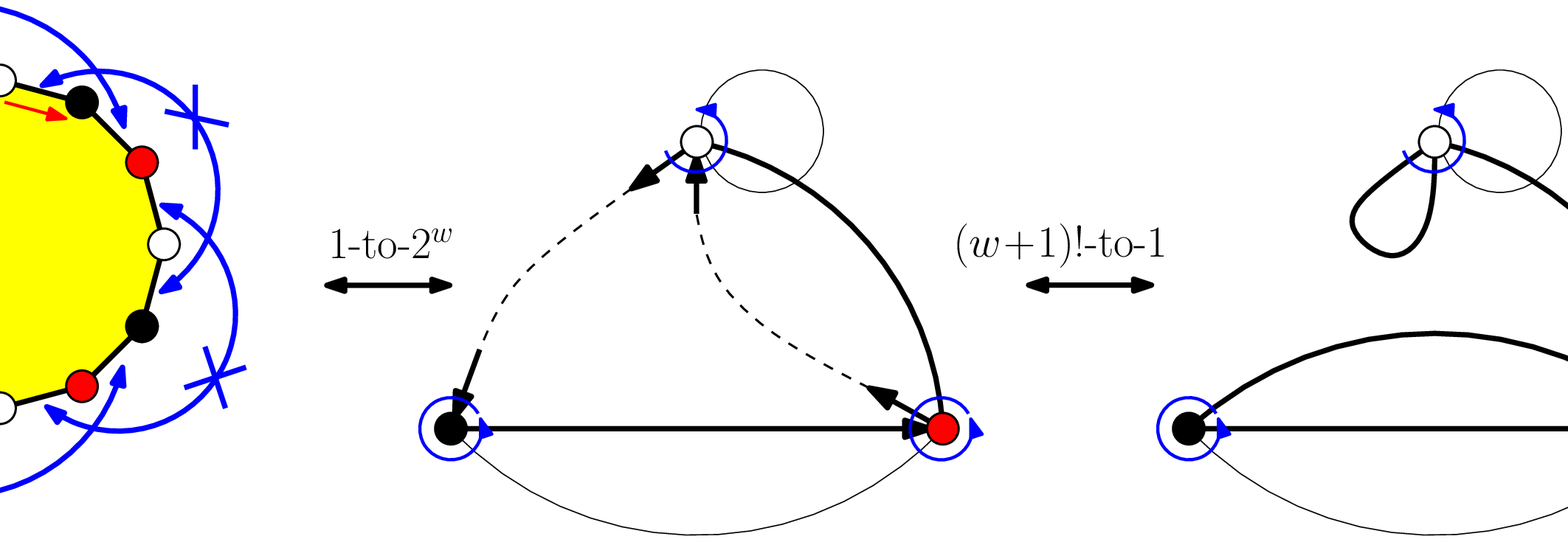}{Some of the steps of the correspondence between $q$-colored unicellular maps and (non-externally-labelled) planar-rooted maps. Here $q=3$ and the planar submap has $w+1=3$ faces.}

\begin{proof}
Let us call \emph{near-Eulerian tree-rooted maps of external weight $w$} a map obtained by adding some non-oriented edges to a near-Eulerian tree of external weight $w$. Near-Eulerian tree-rooted maps are obtained from  compatibly-oriented tree-rooted maps by cutting in two halves the oriented edges not in the spanning tree. This cutting process $\lambda$ is $w!$-to-1 between  compatibly-oriented tree-rooted of external weight $w$ and near-Eulerian tree-rooted maps of external weight $w$ (since there is $w!$ of pairing  $w$ ingoing buds with $w$ outgoing buds). Moreover, the cutting process $\lambda$ is \emph{degree preserving}  (the degree of any vertex is preserved). 

The bijection $\Gam$ between near-Eulerian trees and plane maps can be seen as a $(w+1)$-to-1 correspondence between near-Eulerian trees with external weight $w$ and rooted planar maps with $(w+1)$ faces (since plane maps are planar maps with a marked face). Moreover, the correspondence $\Gamma$ can easily be extended to a degree-preserving $(w+1)$-to-1 correspondence $\Gamma'$ between near-Eulerian tree-rooted maps of external weight $w$ and planar-rooted maps whose planar submap has $w+1$ faces. Indeed, one can just apply the closure $\Gamma$ to the near-Eulerian tree-rooted maps by ignoring the non-oriented edges not in the spanning tree (see Figure~\ref{fig:summary-bijection}).  Hence combining the cutting process $\lambda$ and the closure $\Gamma'$ gives a degree-preserving $(w+1)!$-to-1 correspondence between compatibly-oriented tree-rooted maps of external weight $w$ and  planar-rooted maps whose planar submap has $w+1$ faces. Combining this correspondence with Theorem~\ref{thm:non-orientable} completes the proof.
\end{proof}

Theorem~\ref{thm:non-orientable-bijection} gives the following enumerative results.

\begin{cor}
The number of $q$-colored rooted unicellular maps with $n$ edges on general surfaces is 
\begin{equation}\label{eq:non-orientable-every-color2}
U_{n}(q)=\sum_{r=1}^{n-q+2}\frac{q!\,r!\,P_{q,r}}{2^{r-1}}\,{2n \choose 2q+2r-4}(2n-2q-2r+3)!!,
\end{equation}
where $P_{q,r}$ is the number of unlabeled rooted planar maps with $q$ vertices and $r$ faces. 
Consequently, the series $\displaystyle U(t,w):=\sum_{n,q\geq 0}\frac{U_{n}(q)}{(2n)!}t^nw^q$ and $\displaystyle \QQ(t,w)=\sum_{q,r> 0}\frac{q!\,r!\,P_{q,r}}{(2q+2r-4)!}t^{q+r-2}w^q$ are related by:  
\beq\label{eq:non-orientable-every-color-GF}
 U(t,w)=\frac12\exp\left(\frac{t}2\right)\QQ\left(\frac{t}2,2w\right).
\eeq
\end{cor}

\begin{proof}
By Theorem~\ref{thm:non-orientable-bijection}, the proof of~\eqref{eq:non-orientable-every-color2} amounts to showing that there are $q!P_{q,r}{2n \choose 2q+2r-4}(2n-2q-2r+3)!!$ planar-rooted maps with $n$ edges, vertex set $[q]$ and such that the submap has $r$ faces. One can first choose the planar submap, and then insert the additional edges. There are $q!P_{q,r}$ choice for the planar submap with  vertex set $[q]$ and $r$ faces. The planar submap has $q+r-2$ edges by Euler relation, hence there are $2q+2r-3$ position to insert a half-edge around it. Thus there are ${2n \choose 2q+2r-4}$ ways of inserting the $2n-2q-2r+4$ additional half-edges around the submap and then $(2n-2q-2r+3)!!$ ways of pairing them together. This proves~\eqref{eq:non-orientable-every-color2}.  From this equation, one gets
\begin{eqnarray}
U(t,w)&=&\sum_{n\geq 0 ,q>0}t^nw^q\sum_{r=1}^{n+q-2}\frac{q!\,r!\,P_{q,r}}{2^{n-q+1}(2q+2r-4)!(n-q-r+2)!}\nonumber\\
&=& \frac{1}{2}\sum_{q>0}w^q\sum_{r>0}\frac{q!\,r!\,2^{q}\,P_{q,r}}{(2q+2r-4)!}\sum_{n\geq q+r-2} \frac{t^n}{2^n(n-q-r+2)!}.\nonumber
\end{eqnarray}
Using $\displaystyle \sum_{n\geq q+r-2} \frac{t^n}{2^n(n-q-r+2)!}=\!\left(\frac{t}{2}\right)^{q+r-2}\!\sum_{m\geq 0}\frac{t^m}{2^mm!}=\!\left(\frac{t}{2}\right)^{q+r-2}\!\!\exp\left(\frac{t}2\right)~$ gives~\eqref{eq:non-orientable-every-color-GF}.
\end{proof}

We do not know any closed formula for the numbers $P_{q,r}$. However, standard techniques dating back to Tutte~\cite{Tutte:census4} (recursive decompositions of maps and the quadratic method; see e.g.~\cite{MBM:quadratic}) give the following algebraic equation: 
\beq \label{eq:algebraic-P}
\begin{array}{l}
27\PP^4-(36x+36y-1)\PP^3+(24x^2y+24xy^2-16x^3-16y^3+8x^2+8y^2+46xy-x-y)\PP^2\\
+xy(16x^2+16y^2-64xy-8x-8y+1)\PP-x^2y^2(16x^2+16y^2-32xy-8x-8y+1)=0
\end{array}
\eeq
characterizing the series $\displaystyle \PP\equiv \PP(x,y)=\sum_{q,r> 0}P_{q,r}x^qy^r$. As explained in Section~\ref{sec:recurrence}, the algebraic equation~\eqref{eq:algebraic-P} yields differential equations characterizing the series $\QQ(t,w)$ and $U(t,w)$.\\


\section{Proof of the bijectivity of $\Gam$.}\label{sec:proof-planar}
This section is devoted to the proof of Theorem~\ref{thm:Eulerian-bijection}. Recall that a rooted \emph{plane map} is a rooted map drawn in the plane (this is equivalent to a planar map on the sphere with a marked face). The infinite face of a plane map is called \emph{outer face} and the others are called \emph{inner faces}. A \emph{dual-path for a face} $f$ of a plane map $M$ is a path in the plane starting inside $f$ and ending in the outer face which avoids the vertices of $M$; its \emph{length} is the number of edges of $M$ crossed by the path. The \emph{dual-distance} of a face $f$ is the minimal length of a dual-path for $f$. The dual distance of the inner faces are represented in Figure~\ref{fig:opening}.
The  \emph{dual-distance orientation} of $M$ is defined as the orientation of the subset of edges $e$ of $M$ which are incident to two faces of different dual-distances in such a way that the face with largest dual-distance lies on the left of the oriented edge $e$. The dual-distance orientation is shown in Figure~\ref{fig:opening}.

\begin{lemma}\label{lem:in=out}
The dual-distance orientation of a plane map is such that, at any vertex the numbers of incident ingoing and outgoing half-edges are equal.
\end{lemma}

\begin{proof} Turning around $v$ in counterclockwise direction the dual-distance increases by 1 each time an outgoing half-edge is crossed and decreases by 1 each time an ingoing half-edge is crossed. Hence the number of ingoing and outgoing half-edges crossed must be equal for a complete tour around $v$.
\end{proof}

Let $M$ be a rooted plane map.  For a face $f$ of dual-distance 1, we consider the $k\geq 1$ edges incident to both $f$ and the outer face. By deleting these edges one gets $k$ connected components, one of which $C$ contains the root-vertex. We call \emph{breakable} the edge incident to both $f$ and the outer face that follows $C$ when turning around $f$ in counterclockwise direction. The breakable edges of a map $M$ are represented in Figure~\ref{fig:opening}. 
We now consider the dual-distance orientation of $M$. Observe that all the breakable edges of $M$ are oriented, hence cutting the breakable edges in two halves produces an ingoing and an outgoing bud. Moreover, cutting all the breakable edges of $M$  does not disconnects $M$. Indeed, no cocycle of $M$ contains only breakable edges since a cocycle containing a breakable edge $e$ of a face $f$ also contains another edge $e'$ of $f$ which is either not incident to the outer face (hence not breakable) or incident to the outer face but not breakable (since it is not the breakable edge of $f$). Thus cutting all the  breakable edges of $M$  produces a connected map $M_1$ with some ingoing and outgoing buds and decreases by 1 the dual-distance of every inner face. If $M_1$ has some inner faces, then we can define the breakable edges of $M_1$ as before and consider the map $M_2$ with buds resulting from cutting these edges, etc. We denote by $\Delta(M)$ the rooted plane tree with partial orientation and buds resulting from this ``opening'' process.

\fig{width=\linewidth}{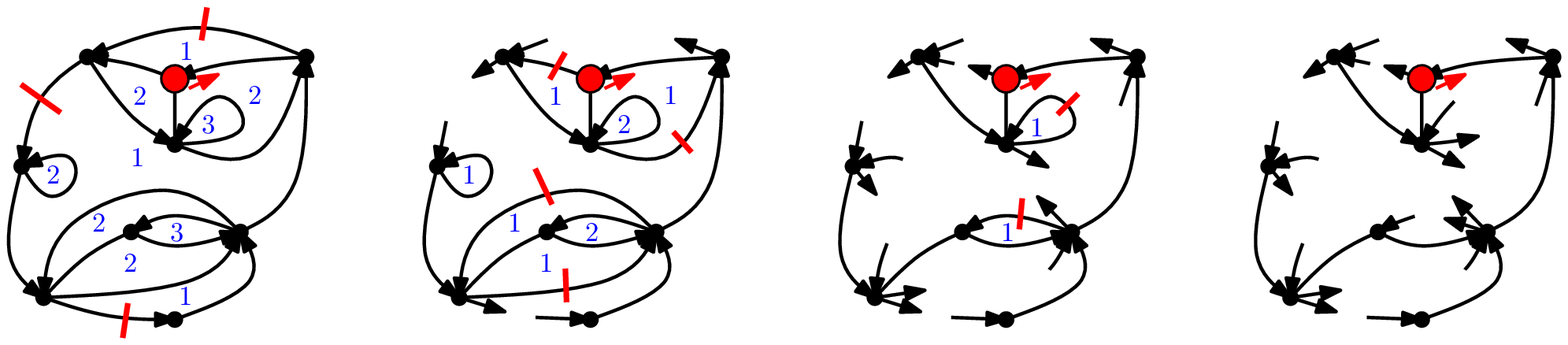}{The opening process $\Delta$. The root-vertex is indicated by a large dot. At each step the dual distance and the breakable edges are indicated.}

\begin{lemma}
For any rooted plane map $M$, the opened map $\Delta(M)$ is a near-Eulerian tree and $\Gamma(\Delta(M))=M$. 
\end{lemma}

\begin{proof} 
Let $T=\Delta(M)$. By Lemma~\ref{lem:in=out}, each vertex of $T$ is incident to as many ingoing and outgoing half-edges. We now consider an edge of $T$, and want to prove that it is not oriented in the wrong direction in $T$. By definition, if $e$ is oriented, then is is incident to two faces of different dual distances in $M$. We consider the step of the opening process $\Delta$ at which both sides of $e$ became equal to the outer face. Just before this step, the edge $e$ was incident to the outer face and to an inner face $f$ of dual-distance 1 (and $e$ had $f$ on its left by definition of the dual-distance orientation). Since $e$ is in $T$, it means that $e$ is not the breakable edge of $f$. Therefore, the definition of the breakable edge of $f$ implies that cutting this breakable edge makes $e$ a bridge oriented toward the component containing the root-vertex. Hence the orientation of $e$ coincides with the orientation  toward the root-vertex of $T$. Thus, $T=\Delta(M)$  is a near-Eulerian tree. 

We now need to show that $\Gamma(T)=M$. Clearly the pairs of outgoing and ingoing buds of $T$ which came from cutting breakable edges of $M$ can all be joined in a non-crossing way (since $M$ is planar). Moreover, it is clear from the definition of the dual-distance orientation that any breakable edge is counterclockwise around~$T$. Thus $\Gamma(T)=M$ by definition of the closure~$\Gamma$.
\end{proof}

To complete the proof of  Theorem~\ref{thm:Eulerian-bijection} it only remains to prove that the closure~$\Gamma$ is injective.  Let $M$ be a plane map and let $T$ be a near-Eulerian tree such that $M=\Gam(T)$. The tree $T$ induces an orientation of a subset of edges of $M$: the subset of oriented edges of $T$ (oriented toward the root-vertex of $T$), together with the set of edges formed during the closure (oriented from the outgoing bud to the ingoing bud). We call this partial orientation the \emph{orientation of $M$ induced by $T$}. We now show that any induced orientation of $M$ is equal to the dual-distance orientation. 

\begin{lemma}\label{lem:dual-orient=induced-orient}
Let $T$ be a near-Eulerian tree and let $M=\Gam(T)$ be its closure. The orientation induced by $T$ coincide with the dual-distance orientation of $M$. 
Moreover, for any inner face $f$ of $M$, there is a dual-path for $f$ of minimal length (equal to the dual-distance of $f$) which does not cross any edge of $T$.
\end{lemma}

\begin{proof} 
We call $T$\emph{-distance} of a face $f$ of $M$ the minimal length of a dual-path for $f$ which does not cross any edge of $T$. We also define the $T$\emph{-distance orientation} of $M$ as the orientation of the subset of edges $e$ of $M$ which are incident to two faces of different $T$-distances in such a way that the face with largest $T$-distance lies on the left of the oriented edge $e$.

We first show that the $T$-distance orientation coincides with the orientation induced by $T$. Let $e$ be an edge not in $T$. Clearly, the edge $e$ separates two faces whose $T$-distances differ by~1. Moreover, the $T$-distance orientation of $e$ coincides with the orientation induced by $T$: in both of these orientations the face with largest $T$-distance is on the left of $e$ (since $e$ is counterclockwise around $T$). We now consider an edge $e$ in the tree $T$. Deleting $e$ gives two subtrees with buds, say  $T_1$ an $T_2$ with $T_1$ containing the root-vertex. Let $d\geq 0$ be the number of outgoing buds in $T_2$ paired with ingoing buds in $T_1$ during the closure of $T$. The number $d'$ of outgoing buds in $T_1$ paired with ingoing buds in $T_2$ is equal to $d$ if $e$ is non-oriented and is equal to $d+1$ if $e$ is oriented in $T$ (since $e$ is oriented toward $T_1$ in this case). Since all the pairing are counterclockwise around $T$, the $T$-distance of the two faces of $M$ incident to $M$ are respectively $d$ and $d'$. Hence the  $T$-distance orientation of $e$ coincides with the orientation induced by $T$: the edge $e$ is oriented if and only if  $d'=d+1$, and in this case the face of largest $T$-distance $d'$ is on the left of $e$. 

We have shown that the orientation induced by $T$ coincides with the $T$-distance orientation. Thus it only remain to show that  $T$-distance is equal to the dual-distance. Let $f$ be a face of $M$ and let $d,d'$ be respectively its dual-distance and its $T$-distance. Clearly, $d\leq d'$.  We now consider a dual-path $P$ of minimal length $d$ for $f$. The dual-path $P$ starts inside the face $f$ which has $T$-distance $d'$ and ends in the outer face which has $T$-distance 0. Moreover, we have shown above that for any edge $e$ of $M$ the difference of $T$-distance between the two faces separated by $e$ is at most 1. Hence, the $T$-distance decreases by at most 1 each time the path $P$ crosses an edge. This shows that $P$ crosses at least $d'$ edges, so that $d\geq d'$.
\end{proof}

We now complete the proof of the injectivity of $\Gamma$. Let $M$ be  a plane rooted map, and let $T$ be a near-Eulerian tree such that $\Gamma(T)=M$. We want to prove that $T=\Delta(M)$. Given Lemma \ref{lem:dual-orient=induced-orient}, it suffices to prove that $T$ and $\Delta(M)$ have the same edges (since the orientation of the edges and buds of $T$ are uniquely determined by the dual-orientation of $M$). We denote $M_0=M$, and  for $i>0$ we denote by $M_i$ the map with buds obtained after $i$ steps of the opening process. We will now show by induction on $i$ that the edges of $T$ are in $M_i$ for all~$i$ (hence  that $T$ has the same edges as $\Delta(M)$). The case $i=0$ is obvious, and to prove the induction step we can suppose that the edges of $T$ are in $M_i$ and show that none of breakable edges of $M_i$ is in $T$. 
 Let $f$ be a face of dual-distance 1 in $M_i$. Let $E$ be the set of edges incident to both $f$ and the outer face of $M_i$. The second statement in Lemma~\ref{lem:dual-orient=induced-orient} shows that at least one edge in $E$ is not in $T$ (because $E$ is the set of edges incident to $f$ and faces of smaller dual-distance, so that any dual-path for $f$ of minimal length crosses one of these edges). We call $e$ the edge which is not in $T$ and observe that the edges in $E\setminus \{e\}$ are all in $T$ (because cutting two edges in $E$ would disconnect $M_i$ hence also $T$). Moreover, the first statement in Lemma~\ref{lem:dual-orient=induced-orient} shows that the dual-distance orientation of the edges in $E\setminus \{e\}$ coincide with the orientation of $T$ toward its root-vertex. Therefore, the only possibility is that $e$ is the breakable edge of $f$. This shows that the breakable edge of $f$ is not in $T$. This proves the induction step, thus $T=\Delta(M)$. This completes the proof of the injectivity, hence bijectivity of $\Gamma$. \hfill $\square$\\


\section{Recursion formula for unicellular maps}\label{sec:recurrence}
In this section we sketch a proof of Ledoux's recursion formula~\eqref{eq:rec-ledoux} starting from~\eqref{eq:non-orientable-every-color-GF}. Some of the calculations require a computer algebra system (or a lot of patience).\\

The number $\hU_n(N)$ of unicellular maps with $n$ edges and vertices colored using some colors in $[N]$ can be seen as a polynomial whose coefficients are the numbers $\eta_v(n)$: $\hU_n(x)=\sum_{v=1}^{n+1} \eta_v(n)x^v$. Hence, proving~\eqref{eq:rec-ledoux} amounts to proving for all $n\geq 2$.
$$ 
\begin{array}{r@{\, }c@{\, }l} 
(n+1)\hU_n(x)&=&(4n-1)(2x-1)\hU_{n-1}(x)+(2n-3)(10n^2-9n+8x-8x^2)\hU_{n-2}(x)\\
&&+5(2n-3)(2n-4)(2n-5)(1-2x)\hU_{n-3}(x)\\
&&-2(2n-3)(2n-4)(2n-5)(2n-6)(2n-7)\hU_{n-4}(x).
\end{array}
$$
Equivalently, one has to prove for all $n\geq 0$,
\begin{equation}\label{eq:QFD}
\begin{array}{r@{\,}c@{\,}l}
-(2n)(2n-1)(2n-2)(n+1)\hV_n(x)+(4n-1)(2n-2)(2x-1)\hV_{n-1}(x)\\+(10n^2-9n+8x-8x^2)\hV_{n-2}(x)+5(1-2x)\hV_{n-3}(x)-2\hV_{n-4}(x)&=&0,
\end{array}
\end{equation}
where $\hV_{k}(x)=\hU_{k}(x)/(2k)!$ for all $k$. 
We now translate~\eqref{eq:QFD} in terms of the numbers $U_{n}(q)$. Recall that $\hV_n(x)=\hU_{n}(x)/(2n)!=\sum_{q\in\NN} {x\choose q} \frac{U_{n}(q)}{(2n)!}$ (because for any positive integer $N$, $\sum_{q\in\NN} {N\choose q} U_n(q)=\hU_n(N)$ counts unicellular maps with $n$ edges and vertices colored using some colors in $[N]$, the index $q$ representing the number of colors really used). 
\begin{lemma}\label{lem:Ntoq}
Let $R_n(x),n\in\NN$ be polynomials in $x$. The polynomial $R_n$ decomposes on the basis ${x\choose q},q\in \NN$ and we denote by $R_{n,q}$ the coefficients: $R_n(x)=\sum_{q\in\NN}{x\choose q} R_{n,q}$. 
Let $c_i(x,y),i=1\ldots k$ be bivariate polynomials. Then, the polynomials $R_n(x),n\in\NN$ satisfy a recursion formula 
\beq\label{eq:Ntoq1}
\textrm{for all } n\geq 0,~~~\quad\quad~ \sum_{i=0}^kc_i(n,x)R_{n-i}(x)=0, 
\eeq
(with the convention $R_n(x)=0$ for $n<0$) if and only if the formal power series $R(t,w)=\sum_{n,q\in\NN}R_{n,q}t^nw^q$ satisfies
\beq\label{eq:Ntoq2}
\sum_{i=0}^kc_i(\De_t,\De_w)(t^iR(t,w))=0,
\eeq
where $\De_t$ and $\De_w$ are the operators on formal power series defined by 
$$\De_t:F(t,w)\mapsto t\deriv{t}F(t,w), ~~~\quad\quad\quad\quad~~\De_w:F(t,w)\mapsto w\deriv{w}((1+w)F(t,w)),$$
and $c_i(\De_t,\De_w)$ is the operator obtained by composing these two commuting operators according to $c_i$.
\end{lemma}

\begin{proof} 
Pascal's rule for binomial coefficients implies that for any polynomial $\QQ(x)=\sum_{q\geq 0}P_q{x\choose q}$, one has $x\QQ(x)=\sum_{q\geq 0}q(P_q+P_{q-1}){x\choose q}$. This relation gives the equivalence between~\eqref{eq:Ntoq1} and a (linear) recursion formula for the numbers $R_{n,q}$ with coefficients in $\CC[n,q]$. Writing this relation in terms of operators acting on $R(t,w)$ gives~\eqref{eq:Ntoq2}.
\end{proof}

By Lemma~\ref{lem:Ntoq}, Formula~\eqref{eq:QFD} is equivalent to a linear differential equation on the power series $\displaystyle U(t,w):=\sum_{n,q\geq 0}\frac{U_{n}(q)}{(2n)!}t^nw^q$. This series is equal to $\frac12\exp\left(\frac{t}2\right)\QQ(\frac{t}2,2w)$ by~\eqref{eq:non-orientable-every-color-GF}, hence by factorizing out $\exp\left(\frac{t}2\right)$ in the differential equation one obtains a (linear) differential equation for $\QQ(\frac{t}2,2w)$.
The differential equation obtained for $\QQ(t,w)$ is:
\barray
&& (6tw+4w^2-36t-7w-6)\QQ(t,w)-(12t^{2}+8tw+7w-25){\frac{\partial}{\partial t}}\QQ(t,w)\\
&& +w(w+2)(8w+6t-7){\frac{\partial}{\partial w}}\QQ(t,w)+(2t^{2}-4tw+37t+9){\frac{\partial^{2}}{\partial t^{2}}}\QQ(t,w)\\
&&-w(w+2)(8t+7){\frac{\partial^{2}}{\partial w\partial t}}\QQ(t,w)+2{w}^{2}(w+2)^{2}{\frac{\partial^{2}}{\partial{w}^{2}}}\QQ(t,w)\\
&& +t(8t+11){\frac{\partial^{3}}{\partial t^{3}}}\QQ(t,w)-4wt(w+2){\frac{\partial^{3}}{\partial w\partial t^{2}}}\QQ(t,w)+2t^{2}{\frac{\partial^{4}}{\partial t^{4}}}\QQ(t,w)=0.
\earray
In terms of coefficients, this equation reads: for all integers $q,r$
\beqarray{eq:rec-hP}
&&\displaystyle 2(q+1)(q+2)\hP_{q,r+3}+6(q+2)\hP_{q+1,r+1}-(q+2)(7+8r)\hP_{q+1,r+2}\\
&&\displaystyle-(4q+4r+11)(q+r+2)(q+2)\hP_{q+1,r+3}-12(r+1)\hP_{q+2,r}\\
&&\displaystyle-2(3q^2+6qr+18q+15r+25-r^2)\hP_{q+2,r+1}\\
&&\displaystyle+(q+r+2)(8qr+8r^2+7q+29r+18)\hP_{q+2,r+2}\\
&&\displaystyle+(q+r+4)(q+r+3)(q+r+2)(2q+2r+5)\hP_{q+2,r+3}=0
\eeqarray
where $\displaystyle \hP_{i,j}=\frac{i!\,j!\,P_{i,j}}{(2i+2j-4)!}$ if $i>0, j>0$ and $\hP_{i,j}=0$ otherwise. 

So far we have proved, using~\eqref{eq:non-orientable-every-color-GF}, that Ledoux's formula~\eqref{eq:rec-ledoux} is equivalent to Equation~\eqref{eq:rec-hP} about the numbers  $P_{q,r}$ of planar maps. We will now prove~\eqref{eq:rec-hP} using the algebraic equation~\eqref{eq:algebraic-P} satisfied by the generating function $\PP(x,y)=\sum_{q,r> 0}P_{q,r}x^qy^r$. First of all, the cases $q<-1$ or $r<-2$ of~\eqref{eq:rec-hP} are trivial. The case $r=-2$ is easily checked using the fact that for all $i>0$, $\hP_{i,1}=\frac{i!P_{i,1}}{(2i-2)!}=\frac{1}{(i-1)!}$ (because $P_{i,1}$ is the Catalan number counting plane trees with $i$ vertices). Now, for $q,r\geq -1$, one can multiply Equation~\eqref{eq:rec-hP} by $\frac{(q+1)!(r+1)!}{(2q+2r+4)!}$. The equation then becomes  
$(2q+2r+4) A_{q,r}=0,$
where $A_{q,r}$ is the coefficient of $x^{q+2}y^{r+1}$ in 
\barray
A(x,y)&=&\left(4 y-2 x-1\right)\left(72 \PP(x,y)-72 x{\frac{\partial}{\partial x}}\PP(x,y)-\left(72 y-2\right){\frac{\partial}{\partial y}}\PP(x,y)\right)\\
&&-72 {x}^{2}{\frac{\partial^{2}}{\partial{x}^{2}}}\PP(x,y)+\left(8 {x}^{2}-6 x-24 xy-56 {y}^{2}+10 y+1\right){\frac{\partial^{2}}{\partial{y}^{2}}}\PP(x,y)\\
&&+2 x\left(4 x-80 y-1\right){\frac{\partial^{2}}{\partial y\partial x}}\PP(x,y)\\
&&+\left(4 x-8 y-1\right)\left(y\left(4 x+48 {y}^{2}-8 y-1\right){\frac{\partial^{3}}{\partial{y}^{3}}}\PP(x,y)+48 {x}^{3}{\frac{\partial^{3}}{\partial{x}^{3}}}\PP(x,y) \right)\\
&&+\left(4 x-8 y-1\right)\left(144 {x}^{2}y{\frac{\partial^{3}}{\partial y\partial{x}^{2}}}\PP(x,y) +x\left(4 x-8 y-1+144 {y}^{2}\right){\frac{\partial^{3}}{\partial{y}^{2}\partial x}}\PP(x,y)\right).
\earray
Since $\PP(x,y)$ is algebraic (of degree 4), all its partial derivatives can be expressed as polynomials (of degree at most 3) in $\PP(x,y)$. In particular, one can express $A(x,y)$ as a polynomial in $\PP(x,y)$. Doing so gives $A(x,y)=0$, hence $A_{q,r}=0$ for all integers $q,r$. Thus,~\eqref{eq:rec-hP} holds for all integers  $q,r$. This completes the proof of Ledoux's recursion formula~\eqref{eq:rec-ledoux}.\hfill$\square$

\section{Additional remarks}\label{sec:conclusion}
The correspondence $\Psi$ stated in Theorem~\ref{thm:non-orientable-bijection} does not allow to count \emph{bipartite} unicellular maps (these maps are related to the characters of the hyperoctahedral group). However, one can establish some enumerative results about these maps by combining Theorem~\ref{thm:non-orientable} with some enumerative results about near-Eulerian trees. In particular, one can write a system of algebraic equations characterizing the generating function $G(x,y,z,u)$ of \emph{bicolored} near-Eulerian trees (with black and white vertices), where the variables $x,y$ correspond respectively to the number of black and white vertices, the variable $z$ correspond to the number of outgoing buds, and the variable $u$  correspond to the difference between the numbers of outgoing buds incident to black and ingoing buds incident to white vertices. The number of bipartite maps with vertices colored in $[p]$ and in $\{p+1,\ldots,p+q\}$ respectively forming a bipartition can then be expressed in terms of the coefficients of $x^py^qz^nu^0$ in $G(x,y,z,u)$, for $n\geq 0$.

In Section~\ref{sec:recurrence} we showed how to use formula~\eqref{eq:non-orientable-every-color} in order to prove the recurrence formula~\eqref{eq:rec-ledoux} for the numbers $\eta_v(n)$ of unicellular maps with $v$ vertices and $n$ edges. The same method should apply to any conjectural recurrence formula, hence one could hope to guess and check a ``simpler'' recurrence formula. One can also perform the steps of Section~\ref{sec:recurrence} in the reverse direction: start from any linear differential equation for the series $\PP(x,y)$ of planar maps and translate it into a recurrence formula for  the numbers $\eta_v(n)$. Alas, we have found no formula simpler than~\eqref{eq:rec-ledoux} by this method.\\

\noindent \textbf{Acknowledgment:} I thank Alejandro Morales for several interesting discussions.


\bibliographystyle{abbrv}
\bibliography{biblio-sum-formula}
\label{sec:biblio}
\end{document}